\documentclass{amsart}

\usepackage{epsf}
\usepackage{amsmath,amssymb,amsfonts,amsthm}
\usepackage{amscd}
\usepackage{mathrsfs}

\usepackage[latin1]{inputenc}

\newtheorem{thm}{Theorem}[section]

\newtheorem{lema}[thm]{Lemma}
\newtheorem{prop}[thm]{Proposition}

\newtheorem{rem}[thm]{Remark}

\newcommand{\uep}{u^\eps}
\newcommand{\bu}{\boldsymbol{u}}
\newcommand{\eps}{\varepsilon}
\newcommand{\R}{\mathbb{R}}
\newcommand{\K}{\mathbb{K}}
\newcommand{\gd}{\nabla}

\newcommand{\cchi}{{\mbox{\raisebox{3pt}{$\chi$}}}}
\newcommand{\lraup}{\relbar\joinrel\relbar\joinrel\relbar\joinrel\rightharpoonup}
\newcommand{\tende}[3]{\mbox{\raisebox{-3pt}{$\begin{CD}{#1}@>>{#2}>{#3}\end{CD}$}}}
\newcommand{\bv}{\boldsymbol{v}}
\newcommand{\bff}{\boldsymbol{f}}
\newcommand{\bs}{\boldsymbol}

\newdimen\authorswidth
\newdimen\maxauthorswidth
\newcommand{\sameauthor}

\newcommand{\nextref}[4]{\bibitem{#1} #2, {\sl #3}, #4.
\global\setbox5=\hbox{#2}
\ifnum\wd5>\maxauthorswidth\authorswidth=\maxauthorswidth \else
\authorswidth=\wd5\fi
\renewcommand{\sameauthor}{\rule{\authorswidth}{0.5truept}}}

\begin{document}

\baselineskip 14pt

\numberwithin{equation}{section}

\title[The nonlinear $N$-membranes problem]{The nonlinear $N$-membranes evolution problem}

\author[J.F. Rodrigues]{Jos\'e Francisco Rodrigues}
\address{CMAF, Department of Mathematics, University of Lisbon,
Av. Prof. Gama Pinto, 2, 1649-003 Lisboa, Portugal}
\email{rodrigue@fc.ul.pt}

\author[L. Santos]{Lisa Santos}
\address{CMAT, Department of Mathematics,
University of Minho, Campus de Gualtar, 4710-057 Braga,
Portugal}
\email{lisa@math.uminho.pt}

\author[J.M. Urbano]{Jos\'e Miguel Urbano}
\address{CMUC, Department of Mathematics, University of Coimbra,
3001-454 Coimbra, Portugal} \email{jmurb@mat.uc.pt}



\begin{abstract}

The parabolic $N$-membranes problem for the $p$-Laplacian and the
complete order constraint on the components of the solution is
studied in what concerns the approximation, the regularity and the
stability of the variational solutions. We extend to the
evolutionary case the characterization of the Lagrange multipliers
associated with the ordering constraint in terms of the
characteristic functions of the coincidence sets. We give continuous
dependence results, and study the asymptotic behavior as
$t\rightarrow\infty$ of the solution and the coincidence sets,
showing that they converge to their stationary counterparts.
\end{abstract}

\maketitle

\vspace{5mm}

\textit{Dedicated to V.A. Solonnikov, on the occasion of his $75$th
birthday, with admiration and friendship.}

\vspace{5mm}

\section{Introduction}

The aim of this work is to analyze the quasilinear parabolic
$N$-system associated with the scalar operator involving the
$p$-Laplacian in the elliptic part
\begin{equation}
Pu_i \equiv \partial_t u_i - \nabla \cdot \left( |\nabla
u_i|^{p-2} \nabla u_i \right) \, , \quad i=1,\ldots,N,
\label{operator}
\end{equation}
with $1<p<\infty$, $\partial_t=\partial / \partial t$ and $\nabla =
(\partial / \partial x_1, \ldots , \partial / \partial x_d)$, in a
space-time cylinder $\Omega_T=\Omega \times (0,T)$, $\Omega \subset
\R^d$, in the case in which the solution
$\boldsymbol{u}=\boldsymbol{u} (x,t)=(u_1,\ldots,u_N)$ has all its
components completely ordered
\begin{equation}
u_1 \geq u_2 \geq \ldots \geq u_N \, , \quad \mbox{a.e. } (x,t) \in
\Omega_T, \label{ordering}
\end{equation}
and subjected to a given nonhomogeneous term
$\boldsymbol{f}=\boldsymbol{f} (x,t)=(f_1,\ldots,f_N)$ and
given boundary conditions. For simplicity, we assume
\begin{equation}
\boldsymbol{u} = \boldsymbol{0}\quad \mbox{on }  \Sigma_T=
\partial \Omega \times (0,T)\qquad\mbox{ and }\qquad \boldsymbol{u}
= \boldsymbol{h}\quad \mbox{on }  \Omega_0=\Omega\times\{0\},
\label{bdry_cond}
\end{equation}
for given Cauchy data $\boldsymbol{h}$.

The time independent case corresponds to the classical $N$-membranes
problem which can be formulated as an elliptic variational
inequality. It has been studied for different types of operators
(see \cite{v1,v2,cv,ars,ars1}) associated with a convex subset of a
Sobolev space determined by the constraint \eqref{ordering}. In the
recent papers \cite{ars,ars1} it has been shown, in particular, that
the $N$-membranes problem can be interpreted as a reaction-diffusion
system with additional discontinuous nonlinearities. In the
evolutionary case \eqref{operator}, it will be shown in this work
that the solution $\boldsymbol{u}$ solves a parabolic system of the
form
\begin{equation} \label{react_diff}
\boldsymbol{P} \boldsymbol{u} = \boldsymbol{f} + \boldsymbol{R}
(x,t,\boldsymbol{u}) \quad \mbox{in} \ \Omega_T ,
\end{equation}
where $\boldsymbol{P} \boldsymbol{u}=(Pu_1,\ldots,Pu_N)$ and each of the components of the nonlinear reaction term
$\boldsymbol{R}$ depends on $(x,t)$ through linear combinations of
the $f_i$, $1\leq i \leq N$, and on $\boldsymbol{u}$ through the
characteristic functions $\cchi_{j,k}=\cchi_{j,k}(x,t)$ of the
$N(N-1)/2$ coincidence sets
\begin{equation}\label{coinc_sets}
I_{j,k} = \left\{ (x,t) \in \Omega_T \, : \, u_j(x,t) = \ldots =
u_k(x,t) \right\} \, , \quad 1\leq j<k\leq N ,
\end{equation}
\textit{i.e.}, $\cchi_{j,k}(x,t)=1$ if $(x,t) \in I_{j,k}$ and
$\cchi_{j,k}(x,t)=0$ otherwise.

We can illustrate the general form of the system \eqref{react_diff} for $N=3$
(see \cite{ars})

\begin{small}
\begin{equation}
\left\{\begin{array}{lllllllll}
Pu_1&\!\!=\!\!&f_1&\!\!+&\!\!\frac12(f_2-f_1)\cchi_{1,2}&&&
\!\!+\!\!&\frac16(2f_3-f_2-f_1)\cchi_{1,3}\vspace{3mm}\\
Pu_2&\!\!=\!\!&f_2&\!\!-\!\!&\frac12(f_2-f_1)\cchi_{1,2}&\!\!+
\!\!&\frac12(f_3-f_2)\cchi_{2,3}&\!\!+\!\!&\frac16(2f_2-f_1-f_3)\cchi_{1,3}\vspace{3mm}\\
Pu_3&\!\!=\!\!&f_3&&&\!\!-\!\!&\frac12(f_3-f_2)\cchi_{2,3}&\!\!+\!\!&\frac16(2f_1-f_2-f_3)\cchi_{1,3}
\end{array}\right.
\label{sistemachi}
\end{equation}
\end{small}

\noindent which contains the simpler case $N=2$, that corresponds to
the first two equations with $\cchi_{2,3}\equiv 0$ and
$\cchi_{1,3}\equiv 0$, in which case the third equation is
independent of the first two. Noting that, in general, $\cchi_{j,k}
= \cchi_{j,j+1} \cchi_{j+1,j+2} \ldots \cchi_{k-1,k}$, for $1 \leq j
< k \leq N$, in \eqref{sistemachi} the last terms containing
$\cchi_{1,3} = \cchi_{1,2} \cchi_{2,3}$ are in fact doubly nonlinear
in $\boldsymbol{u}$. This introduces additional difficulties in
analyzing the stability of the system with respect to the
perturbation of the data. In fact, in section 3, we show that the
sufficient conditions on the averages of the components of
$\boldsymbol{f}$, obtained in \cite{ars1} for the stability of the
coincidence sets $I_{j,k}$ in the stationary problem, extend to the
parabolic case as well. In particular, for $N=3$, they take the form
$$f_1\neq f_2, \qquad f_2\neq f_3,\qquad f_1\neq\frac{f_2+f_3}2,
\qquad f_3\neq\frac{f_1+f_2}2\qquad \mbox{a.e. in }\Omega_T.$$

We notice that the stability result on the $\cchi_{j,k}$ is not a
direct consequence of the stability of the solution $\boldsymbol{u}$
with respect to the data $\boldsymbol{f}$ and $\boldsymbol{h}$,
which can, however, be obtained by direct variational methods, as we
also show in subsection \ref{24}.

Classical monotonicity methods (see \cite{l}, for example) or the
theory of accretive operators and evolution inclusions in Banach
spaces (see \cite{dm}, \cite{s}, \cite{z} or \cite{ao} and their
bibliography) are directly applicable and yield general results on
the existence of solutions to our problem, when formulated as a
variational inequality in the convex set associated with the
constraints \eqref{ordering}. In section 2, we introduce  an
approximation of the variational inequality formulation and we
obtain directly useful \textit{a priori} estimates for the existence
of solutions. We remark that we assume the $p'$-integrability of
$\boldsymbol{f}$ and rely on the $p$-integrability of a compatible
$\boldsymbol{h}$ and its derivatives, but we do not require the
boundedness of $\boldsymbol{h}$ nor of the variational solution
globally in $\overline{\Omega}_T$.

Considering the relation of the upper and lower membranes (in
particular, the two-membrane problem) with the obstacle problem and
of the inner membranes of the $N$-problem, with $N\geq 3$, with the
two--obstacles problem, we apply the dual estimates for unilateral
parabolic problems (see \cite{ct}, \cite{dm} or \cite{d}) to obtain
Lewy-{\mbox{-Stampacchia} type inequalities
\begin{equation}
\bigwedge_{j=1}^i f_j \leq P u_i \leq \bigvee_{j=i}^N f_j \qquad
\mbox{a.e. in } \Omega_T, \quad i=1,\ldots,N, \label{ls}
\end{equation}
for the parabolic operator \eqref{operator}. Here we use the
notation
$$\bigvee_{i=1}^k \xi_i =\xi_1 \vee \ldots \vee \xi_k = \sup \{ \xi_1,
\ldots,\xi_k \} \quad \mbox{and} \quad \bigwedge_{i=1}^k \xi_i =
\xi_1 \wedge \ldots \wedge \xi_k = \inf \{ \xi_1,\ldots,\xi_k \}$$
and we also denote $\xi^+=\xi\vee 0$ and $\xi^-=-(\xi\wedge 0)$.

We also show how the estimates on $P u_i$ imply that the variational
solution to the $N$-membranes problem solves a.e. a system of the
type \eqref{react_diff}, for an explicit $\boldsymbol{R}$ with the
same $p'$-integrability as $\bff$, extending the analogous result
obtained in \cite{ars1} for the stationary problem. This implies, in
particular when $\bff$ is bounded, the H\"{o}lder continuity of the
solution and of its gradient. In fact, this is an immediate
consequence of known estimates for the parabolic operator
\eqref{operator}, even without knowing the explicit form of
$\boldsymbol{R}$, as we observe in section 2. Even for the linear
case $p=2$, for which we can apply Solonnikov's estimates in
$W^{2,1}_p(\Omega_T)$, the regularity obtained here for the solution
of the evolutionary $N$-membranes problem is new.

In section 3, we study the asymptotic convergence, when $t
\rightarrow \infty$, of the solution $\boldsymbol{u} (t)$ to the
corresponding solution of the stationary problem of \cite{ars1} in
$\boldsymbol{L}^2(\Omega_T)$  (here we denote
$\boldsymbol{L}^2(\Omega_T)=\left[L^2(\Omega_T)\right]^N$), in the
case $p \geq 2$. We show how a modest convergence of the solution,
obtained as in \cite{simon}, also implies the asymptotic
stabilization of the evolution coincidence sets towards the
stationary ones, under a natural nondegeneracy assumption identified
in \cite{ars1}.

Finally, we observe that most results still hold, with suitable
adaptations, for more general quasilinear parabolic scalar operators
$$Pu = \partial_t u - \nabla \cdot \left( a(x,t,\nabla u) \right), $$
in particular, for strongly monotone vector fields $a(\cdot,\xi)$,
with $p$-structure as in \cite{ars1}, as well as more general data
$\bff$ in $L^q(0,T;\boldsymbol{L}^r(\Omega))$ (see \cite{bdgo}).

For simplicity of presentation, we limit ourselves here to the case
of the $p$-\mbox{-Laplacian} with homogeneous Dirichlet data, i.e., we
consider only variational solutions in the usual Sobolev space
$\boldsymbol{W}^{1,p}_0(\Omega)=\left[W^{1,p}_0(\Omega)\right]^N$,
for $1<p<\infty$. The case of a time-dependent Dirichlet boundary
condition is more delicate and will be considered in \cite{rsu}.

\vspace{5mm}

\section{Approximation and regularity of variational solutions}

Let $\Omega \subset \mathbb{R}^d$ be a bounded domain with a
Lipschitz boundary, let $T>0$, and define the space-time domain
$\Omega_T:=\Omega \times (0,T)$, with parabolic boundary $\partial_p
\Omega_T:=\Sigma_T \cup  \Omega _0$. We use $N$-vectorial notation for vector fields
$$\boldsymbol{w}:=(w_1,\ldots,w_N) \in \mathbb{R}^N$$
and function spaces $\boldsymbol{F} := \left[ F \right]^N$. For
$1<p<\infty$, define the differential operator
$$\nabla_p \boldsymbol{w} = \left( \nabla_p w_1, \ldots , \nabla_p w_N\right),
\quad \nabla_p w_i := \left| \nabla w_i \right|^{p-2} \nabla w_i
\quad\mbox{ and }\quad\Delta_p\boldsymbol{w}=\nabla\cdot\nabla_p
\boldsymbol{w}.$$

We assume the data satisfy
\begin{equation}\label{f}
\boldsymbol{f} \in \boldsymbol{L}^{p'} (\Omega_T) \qquad\mbox{ and }\qquad
\boldsymbol{h}\in\mathbb{K}\cap\boldsymbol{L}^2(\Omega),
\end{equation}
where $p'=p/(p-1)$ and $\K$ is the closed convex subset of
$\boldsymbol{W}^{1,p}_0(\Omega)$ defined by
\begin{equation}
\mathbb{K} = \left\{ \boldsymbol{v} \in
\boldsymbol{W}^{1,p}_0(\Omega) \: : \: v_1 \geq \ldots \geq v_N, \
\mbox{a.e. in} \ \Omega\right\} . \label{convexo}
\end{equation}

\vspace{3mm}

\subsection{Variational formulation of the problem}

The evolutive $N$-membranes problem for the $p$-Laplace operator
consists in finding a vector field
$\boldsymbol{u}=\boldsymbol{u}(x,t)$ such that
\begin{equation}
\boldsymbol{u} \in  L^p \left( 0,T; \boldsymbol{W}^{1,p} _0(\Omega)
\right) \cap C \left( [0,T]; \boldsymbol{L}^2(\Omega) \right),
\label{def1}
\end{equation}
\begin{equation}
\partial_t \boldsymbol{u} \in L^{p^{\prime}} \big( 0,T; \boldsymbol{W}^{-1,p^{\prime}}
(\Omega) \big),\label{def2}
\end{equation}
\begin{equation}\label{def3}
\boldsymbol{u}(t) \in \mathbb{K}, \ \textrm{a.e.} \ t \in (0,T),
\quad \boldsymbol{u}(0) = \boldsymbol{h}  \in
\boldsymbol{L}^2(\Omega),
\end{equation}
and, for a.e. $t \in (0,T)$ and all $\boldsymbol{v} \in \mathbb{K}$,
\begin{equation}\label{ineq_var}
\left\langle \partial_t \boldsymbol{u} (t) ,
\boldsymbol{v}-\boldsymbol{u}(t)\right\rangle + \int_\Omega \nabla_p
\boldsymbol{u}(t) : \nabla \left(
\boldsymbol{v}-\boldsymbol{u}(t)\right) \geq \int_\Omega
\boldsymbol{f}(t) \cdot \left(
\boldsymbol{v}-\boldsymbol{u}(t)\right).
\end{equation}
Here, $\langle \cdot , \cdot \rangle$ denotes the sum of the $N$
duality pairings in $\bs W^{-1,p^{\prime}} (\Omega) \times \bs W_0^{1,p}
(\Omega)$ of the components of the vector fields, and $A:B$ denotes
the scalar product of the matrices $A$ and $B$.

We observe that, by a simple comparison argument, there exists at
most one solution of \eqref{def1}--\eqref{ineq_var}, the variational
inequality formulation of the evolutionary $N$-membranes problem.

\vspace{3mm}

\subsection{The approximating problem} \label{22}

We approximate the variational inequality \eqref{ineq_var} using a
bounded penalization. For that purpose, for each $\varepsilon>0$,
let $\theta_\varepsilon$ be the real function defined in $[-\infty,
+\infty]$ by
$$\theta_\varepsilon (\theta)=\left\{
\begin{array}{cll}
-1 \ & \ \textrm{if}\ & \theta \leq -\eps\vspace{2mm}\\
\theta/\varepsilon \ & \ \textrm{if}\ & -\eps< \theta < 0\vspace{2mm}\\
0 \ & \ \textrm{if}\ & \theta \geq 0.
\end{array} \right.
$$
The approximating penalized problem is the system of boundary value
problems defined as follows:
\begin{equation}
\label{prob aprox}
 \left\{\begin{array}{l}
 P u_i^\varepsilon  + \xi_i\theta_\varepsilon \left( u_i^\varepsilon -
u_{i+1}^\varepsilon \right)- \xi_{i-1} \theta_\varepsilon \left(
u_{i-1}^\varepsilon - u_i^\varepsilon  \right)  =  f_i\quad
  \textrm{ in } \Omega_T\vspace{3mm}\\
u_i^\varepsilon =0\quad\textrm{ on }\Sigma_T\qquad\mbox{ and }\qquad
u_i^\varepsilon=h_i\quad\textrm{ on }\Omega_0\vspace{1mm}
\end{array}
\right.
\end{equation}
with $i=1,\ldots,N$, and the convention  $u_0\equiv +\infty$ and
$u_{N+1} \equiv -\infty$, where for $i=1,\ldots,N$,
\begin{equation}\label{xiis}
\xi_0=\max\Big\{\frac{f_1+\cdots+f_i}{i}:i=1,\ldots,
N\Big\},\qquad\xi_i=i\,\xi_0-(f_1+\cdots+f_i),
\end{equation}  (see
\cite{ars1}). Notice that, for $i=1,\ldots,N$, we have $\xi_i\ge 0$ and $\xi_i\in L^{p'}(\Omega)$.

\begin{lema} \label{Tmonotone}Using the convention $v_0=+\infty$ and $v_{N+1}=-\infty$, the operator
$$\langle B\boldsymbol{v},\boldsymbol{w}\rangle=\sum_{i=1}^N\int_\Omega \left(\xi_i \theta_\varepsilon \left(v_i - v_{i+1}
\right)- \xi_{i-1} \theta_\varepsilon \left( v_{i-1} - v_i
\right)\right)w_i,\qquad\boldsymbol{v},\,\boldsymbol{w}\in\,\boldsymbol{W}^{1,p}_0(\Omega),$$
is T-monotone, i.e.,
$$\langle
B\boldsymbol{v}-B\boldsymbol{w},(\boldsymbol{v}-\boldsymbol{w})^+\rangle\ge
0,\qquad\forall\,\boldsymbol{v},\boldsymbol{w}\in\boldsymbol{W}^{1,p}_0(\Omega).$$
\end{lema}

\begin{proof} Since we can rewrite
$$\langle B\boldsymbol{v},\boldsymbol{w}\rangle=\sum_{i=1}^{N-1}\int_\Omega \xi_i
\theta_\varepsilon \left(v_i - v_{i+1} \right)(w_i-w_{i+1}),$$ it is
enough to observe that
\begin{multline}
\langle
B\boldsymbol{v}-B\boldsymbol{w},(\boldsymbol{v}-\boldsymbol{w})^+\rangle=\\
\nonumber\displaystyle{\hspace{1cm}\sum_{i=1}^{N-1}\int_\Omega
\xi_i\Big(\theta_\varepsilon(v_i-v_{i+1})-\theta_\varepsilon(w_i-w_{i+1})\Big)\,
\big((v_{i}-w_i)^+-(v_{i+1}-w_{i+1})^+\big)}.
\end{multline}
As $\xi_i\ge 0$, for $i=1,\ldots,N$ and $\theta_\varepsilon$ is
monotone nondecreasing, the conclusion follows.
\end{proof}

\begin{prop}\label{prop22}
Under assumption {\em (\ref{f})}, the approximating problem
\eqref{prob aprox} has a unique solution $(u_1^\varepsilon, \ldots,
u_N^\varepsilon) \in  L^p ( 0,T;\boldsymbol{ W}^{1,p}_0 (\Omega) )
\cap C\left( [0,T]; \boldsymbol{L}^2(\Omega) \right)$ such that
\begin{equation}
u_i^\varepsilon \leq u_{i-1}^\varepsilon+\eps, \quad i=2,\ldots,N.
\label{sol_aprox}
\end{equation}
\end{prop}

\begin{proof}
The existence and uniqueness follow, respectively, from standard
results concerning monotone operators and comparison (see \cite{l}
or \cite{z}), for instance, using the Faedo-Galerkin approximation.
We notice that, since $\bff\in L^{p'}(\Omega_T)\subset
L^{p'}(0,T;\boldsymbol{W}^{-1,p'}(\Omega))$, we obtain, in
particular, that $\partial_t\boldsymbol{u}^\eps\in
L^{p'}(0,T;\boldsymbol{W}^{-1,p'}(\Omega))$.

To prove inequality \eqref{sol_aprox}, multiply both the $i$-th and
the $(i-1)$-th equations by $(u_i^\varepsilon - u_{i-1}^\varepsilon
-\eps)^+$, subtract and integrate over $\Omega$, obtaining
  \begin{alignat*}{1}
\frac{1}{2} \frac{d}{dt} \int_\Omega \left| (u_i^\varepsilon
-\right.&\left. u_{i-1}^\varepsilon-\eps )^+ \right|^2+\int_\Omega
\Big(\gd_p\uep_i-\gd_p\uep_{i-1}\Big)
\cdot\gd(u_i^\eps-u_{i-1}^\eps-\eps)^+\vspace{1mm}\\
=&\int_\Omega\left[\big(f_i-\xi_i\theta_\eps(\uep_i-\uep_{i+1})+\xi_{i-1}\theta_\eps(\uep_{i-1}-\uep_{i}-\eps)\big)
(\uep_i-\uep_{i-1})^+\right.\vspace{1mm}\\
&\
\left.-\big(f_{i-1}-\xi_{i-1}\theta_\eps(\uep_{i-1}-\uep_{i})+\xi_{i-2}\theta_\eps(\uep_{i-2}-\uep_{i-1})\big)
(\uep_i-\uep_{i-1})^+\right]\vspace{1mm}\\
\le&\int_\Omega\big((f_i-f_{i-1})+(\xi_i-\xi_{i-1})-(\xi_{i-1}-\xi_{i-2})\big)(\uep_i-\uep_{i-1}-\eps)^+\le
0.
 \end{alignat*}

Integrating between $0$ and $t$, using the fact that
$h_1\ge\cdots\ge h_N$ and the inequality $$\displaystyle{\int_\Omega
\Big(\gd_p\uep_i-\gd_p\uep_{i-1}\Big)
\cdot\gd(u_i^\eps-u_{i-1}^\eps-\eps)^+}\ge 0,$$ we get
\begin{equation}
\frac12\int_\Omega\left[\big(\uep_i(t)-\uep_{i-1}(t)-\eps\big)^+\right]^2\le
0,
\end{equation}
and so $\uep_i\le\uep_{i-1}+\eps$ a.e. in $\Omega_T$.
\end{proof}

\vspace{3mm}

\subsection{Existence of variational solutions}

The proof of the existence of solution for the variational
inequality (\ref{ineq_var}) will be done passing to the limit in
$\eps\rightarrow 0$ on the sequence of approximating solutions
$\bu^\eps$, by using the following {\it a priori } estimates that
can be rigorously obtained through the respective Faedo-Galerkin
approximations.

\begin{prop} Under assumption {\em (\ref{f})}, the solution of the approximating problem {\em (\ref{prob
aprox})} satisfies the following estimates, for a nonnegative
constant $C$, independent of $\eps$:
\begin{equation}\label{210}
\|\uep_i\|_{L^\infty(0,T;L^2(\Omega))}+\|\gd\uep_i\|_{L^p(\Omega_T)}\le
C,
\end{equation}
\begin{equation}\label{211}
\|\partial_t\uep_i\|_{L^{p'}(0,T;W^{-1,p'}(\Omega))}\le
C,
\end{equation}
\begin{eqnarray}\label{215}
 \left\{\begin{array}{cllll}
 f_1&\le&P\uep_1&\le& f_1+\xi_1\qquad\qquad\\
\vdots& &&&\vdots\\
 f_i-\xi_{i-1}&\le&P\uep_i&\le&f_i+\xi_i\qquad(i=2,\ldots,N-1)\\
\vdots& &&&\vdots\\
   f_N-\xi_{N-1}&\le&P\uep_{N}&\le& f_N\qquad\mbox{ a.e. in }\Omega_T.\end{array}\right.
\end{eqnarray}
\end{prop}
\begin{proof} For each $i=1,\ldots, N$, we easily conclude (\ref{215}) from (\ref{prob aprox}) in the form
$$P\uep_i=f_i+g^\eps_i\qquad\mbox{ in }\Omega_T,$$
where
\begin{equation}\label{gieps} g_i^\varepsilon = \xi_{i-1}
\theta_\varepsilon \left( u_{i-1}^\varepsilon - u_i^\varepsilon
\right) -\xi_i \theta_\varepsilon \left( u_i^\varepsilon -
u_{i+1}^\varepsilon \right)
\end{equation}
is uniformly bounded in $\boldsymbol{L}^{p'}(\Omega_T)$.

Then, multiplying each equation in (\ref{prob aprox}) by $\uep_i$
and integrating on $\Omega_t=\Omega\times\,(0,t)$, we get
\begin{equation*}
\frac12\int_\Omega |\uep_i(t)|^2+\int_{\Omega_t}|\gd\uep_i|^p
\le\int_{\Omega_t}\big(f_i+g^{\eps}_i\big)\uep_i+\frac12\int_\Omega|\uep_i(0)|^2.
\end{equation*}

Using Poincaré inequality, we find
\begin{equation}\label{estimsimples}
\int_\Omega|\uep_i(t)|^2+\int_{\Omega_t}|\gd\uep_i|^p\le
C_0,
\end{equation}
where the constant $C_0$ only depends on
$\|\boldsymbol{h}\|_{\boldsymbol{L}^2(\Omega)}$ and
$\|\boldsymbol{f}\|_{\boldsymbol{L}^{p'}(\Omega_T)}$. Hence, from
(\ref{estimsimples}), we immediately obtain (\ref{210}). So
\begin{equation}\label{aa}\Delta_p\uep_i\ \mbox{ is bounded in
}  L^{p'}(0,T;W^{-1,p'}(\Omega))\mbox{ independently of }\eps,\end{equation}
and we  conclude
(\ref{211}) by recalling (\ref{215}).
\end{proof}

\begin{thm}\label{teo24}
Under assumption {\em (\ref{f})}, the problem {\em
(\ref{def3})-(\ref{ineq_var})} has a unique variational solution
$\boldsymbol{u}$ in the class {\em (\ref{def1})-(\ref{def2})}.

In addition, $\boldsymbol{u}^\eps\rightarrow\bu$ strongly in
$L^{p}(0,T;\boldsymbol{W}^{1,p}_0(\Omega))$ and
 \begin{equation}
\label{delta}\boldsymbol{P}\boldsymbol{u}^\eps\rightharpoonup
\boldsymbol{P}\boldsymbol{u} \qquad\mbox{ in }\qquad
\boldsymbol{L}^{p'}(\Omega_T)-\mbox{weak}.
\end{equation}
\end{thm}
\begin{proof} If $\{\boldsymbol{u}^\eps\}_\eps$ is a sequence of solutions of the approximating
problems (\ref{prob aprox}), by the {\em a priori} estimates
(\ref{210}) and (\ref{211}), we can extract a subsequence such that,
as $\eps\rightarrow 0$,
$$\boldsymbol{u}^\eps\rightharpoonup \boldsymbol{u}\qquad\mbox{ in }\qquad
L^p(0,T;\boldsymbol{W}^{1,p}_0(\Omega))-\mbox{weak},$$
$$\partial_t\boldsymbol{u}^\eps\rightharpoonup \partial_t\boldsymbol{u}
\qquad\mbox{ in }\qquad
L^{p'}(0,T;\boldsymbol{W}^{-1,p'}(\Omega))-\mbox{weak},$$ and, by compactness,
also $\boldsymbol{u}^\eps\rightarrow \boldsymbol{u}$ strongly in
$\boldsymbol{L}^p(\Omega_T)$.

Let $\boldsymbol{v}\in L^p(0,T;\boldsymbol{W}^{1,p}_0(\Omega))$ be
such that $\partial_t\boldsymbol{v}\in
L^{p'}(0,T;\boldsymbol{W}^{-1,p'}(\Omega))$,
$\boldsymbol{v}(t)\in\K$, for a.e. $t\in\,(0,T)$, and
$\boldsymbol{v}(0)=\boldsymbol{h}$. As $\langle
B\boldsymbol{v}(t),\boldsymbol{v}(t)-\bu(t)\rangle= 0$, we have
\begin{equation*}
\langle\partial_t\boldsymbol{u}^\eps,\boldsymbol{v}-\boldsymbol{u}^\eps\rangle
+\int_{\Omega_T}\gd_p\boldsymbol{u}^\eps:\gd\big(\boldsymbol{v}-\boldsymbol{u}^\eps\big)
\ge\int_{\Omega_T}\boldsymbol{f}
\cdot\big(\boldsymbol{v}-\boldsymbol{u}^\eps\big)
\end{equation*}

It follows from integration by parts that
\begin{equation*}
\langle\partial_t\boldsymbol{u}^\eps,\boldsymbol{v}-\boldsymbol{u}^\eps\rangle=
\langle\partial_t\bv,\bv-\boldsymbol{u}^\eps\rangle-\frac12\int_\Omega|
\bu^\eps(T)-\bv(T)|^2
\end{equation*}
and, using the monotonicity, we get
\begin{alignat*}{1}
\langle\partial_t\boldsymbol{v},\boldsymbol{v}-\boldsymbol{u}^\eps\rangle
+\int_{\Omega_T}\gd_p\boldsymbol{v}:\gd\big(\boldsymbol{v}-\boldsymbol{u}^\eps\big)
\ge &\int_{\Omega_T}\boldsymbol{f}
\cdot\big(\boldsymbol{v}-\boldsymbol{u}^\eps\big)+\frac12\int_\Omega|
\bu^\eps(T)-\bv(T)|^2\\
\ge &\int_{\Omega_T}\boldsymbol{f}
\cdot\big(\boldsymbol{v}-\boldsymbol{u}^\eps\big).
\end{alignat*}
Letting $\eps\rightarrow 0$, we obtain
\begin{equation}
\label{monov}
\langle\partial_t\boldsymbol{v},\boldsymbol{v}-\boldsymbol{u}\rangle
+\int_{\Omega_T}\gd_p\boldsymbol{v}:\gd\big(\boldsymbol{v}-\boldsymbol{u}\big)\\
\ge\int_{\Omega_T}\boldsymbol{f}
\cdot\big(\boldsymbol{v}-\boldsymbol{u}\big).
\end{equation}

Now, let $\boldsymbol{w}=\bu+\theta(\bv-\bu)$, $\theta\in\,(0,1]$.
The verification that $\boldsymbol{w}$ can be chosen as test
function in (\ref{monov}) is immediate. So,
\begin{equation*}
\langle
\partial_t\bu+\theta\partial_t(\bv-\bu),\theta(\bv-\bu)\rangle+
\int_{\Omega_T}\nabla_p\big(\bu+\theta(\bv-\bu)\big):\theta\nabla(\bv-\bu)
\ge\int_{\Omega_T}\theta\boldsymbol{f}\cdot(\bv-\bu).
\end{equation*}

Dividing both members by $\theta$ and letting $\theta\rightarrow 0$, we see that $\bu$ solves
the problem
$$\langle
\partial_t\bu,\bv-\bu\rangle+\int_{\Omega_T}\nabla_p\bu:\nabla(\bv-\bu)
\ge\int_{\Omega_T}\boldsymbol{f}\cdot(\bv-\bu),
$$
for all $\bv$ such that $\bv\in
L^p(0,T;\boldsymbol{W}^{1,p}_0(\Omega))$, $\bv(t)\in\K$ for a.e.
$t\in (0,T)$ and $\bv(0)=\boldsymbol{h}$. Using standard arguments
(see \cite{l}), also
$$\langle
\partial_t\bu(t),\bv(t)-\bu(t)\rangle+\int_{\Omega}\nabla_p\bu(t):\nabla(\bv(t)-\bu(t))
\ge\int_{\Omega}\boldsymbol{f}(t)\cdot(\bv(t)-\bu(t)),
$$
for a.e. $t\in (0,T)$, for all $\bv$ such that $\bv\in
L^p(0,T;\boldsymbol{W}^{1,p}_0(\Omega))$ and $\bv(t)\in\K$.

In order to conclude (\ref{delta}) it is sufficient to recall the
estimates (\ref{215}) for $P\uep_i$ and that
$\nabla_p\uep_i\rightarrow\nabla_p u_i$ in an appropriate sense. In
fact, recalling (\ref{gieps}) and using equation (\ref{prob aprox}),
we conclude that
$$\limsup_{\eps\rightarrow 0}\int_{\Omega_T}\gd_p\boldsymbol{u}^\eps\cdot\gd(\boldsymbol{u}^\eps-\boldsymbol{u})$$
$$\le \limsup_{\eps\rightarrow
0}\left[\int_{\Omega_T}(\bff+\boldsymbol{g}^\eps)\cdot(\boldsymbol{u}^\eps-\boldsymbol{u})-
\langle\partial_t\boldsymbol{u}^\eps,\boldsymbol{u}-\boldsymbol{u}^\eps\rangle\right]=0.$$

By well-known results (see, for instance, \cite{bm}) this is
sufficient to show that
$\boldsymbol{u}^\eps\rightarrow\boldsymbol{u}$ strongly in
$L^p(0,T;\boldsymbol{W}^{1,p}_0(\Omega))$  (notice that
$\boldsymbol{g}^\eps \rightharpoonup \boldsymbol{g}$ weakly in
$\boldsymbol{L}^{p'}(\Omega_T)$, for some~$\boldsymbol{g}$).
\end{proof}

\begin{rem}
If we assume also that
$\boldsymbol{f}\in\boldsymbol{L}^2(\Omega_T)$, which is a
consequence of {\em (\ref{f})} if $1<p\le 2$, the Faedo-Galerkin
approach yields directly the regularity
\begin{equation}
\boldsymbol{u}\in H^1(0,T;\boldsymbol{L}^2(\Omega_T))\cap
L^\infty(0,T;\boldsymbol{W}^{1,p}_0(\Omega))
\end{equation}
through multiplication of {\em (\ref{prob aprox})} by
$\partial_t\uep_i$.
\end{rem}

\vspace{3mm}

\subsection{Strong continuous dependence}\label{24}

\begin{thm} \label{tap}
Let $\boldsymbol{u}^*$ be the variational solution to {\em (\ref{def3})-(\ref{ineq_var})}
corresponding to data  $\boldsymbol{f}^*$ and $\boldsymbol{h}^*$ satisfying also {\em (\ref{f})} and
denote
$$\eps^*\equiv\, \|\boldsymbol{f}^*-\boldsymbol{f}\|^{q}_{\boldsymbol{L}^{q}(\Omega_T)}+
\|\boldsymbol{h}^*-\boldsymbol{h}\|^{2}_{\boldsymbol{L}^{2}(\Omega)},$$
with $q=p'\wedge 2$ {\em (}i.e. $q=p'$ if $p>2$ and $q=2$ if $p\le
2${\em )}. Then there exists a positive constant $c=c(T,p)$ such
that
\begin{equation}
\sup_{0<t<T}\int_\Omega|\bu^*(t)-\bu(t)|^2+\int_{\Omega_T}|\nabla(\bu^*-\bu)|^p\le
c\,\eps^*\qquad\mbox{if $p\ge 2$,}
\end{equation}
\begin{equation}\label{p<2}
\sup_{0<t<T}\int_\Omega|\bu^*(t)-\bu(t)|^2+\Big(\int_{\Omega_T}|\nabla(\bu^*-\bu)|^p\Big)^{\frac2p}\le
c\,\eps^*\qquad\mbox{
if $1<p<2$}.
\end{equation}
\end{thm}
\begin{proof}
Let $\boldsymbol{v}=\boldsymbol{u}^*(t)$ in (\ref{ineq_var}) with
data $\boldsymbol{f}$ and $\boldsymbol{h}$, and
$\boldsymbol{v}=\boldsymbol{u}(t)$ in (\ref{ineq_var}) with data
$\boldsymbol{f}^*$ and $\boldsymbol{h}^*$. In the latter case, we
have
\begin{multline}\label{mutuiep}
\langle\partial_t
\boldsymbol{u}^*(t),\boldsymbol{u}(t)-\boldsymbol{u}^*(t)\rangle+\int_\Omega
\gd_p\boldsymbol{u}^*(t):\gd(\boldsymbol{u}(t)-\boldsymbol{u}^*(t))\\
\ge\int_\Omega\boldsymbol{f}^*(t)\cdot(\boldsymbol{u}(t)-\boldsymbol{u}^*(t)).
\end{multline}
From (\ref{ineq_var}) and (\ref{mutuiep}), integrating between $0$
and $t$, we obtain
\begin{multline}
\label{mutumuiep}
\frac12\int_\Omega|\boldsymbol{u}^*(t)-\boldsymbol{u}(t)|^2+\int_0^t\int_\Omega
(\gd_p\boldsymbol{u}^*-\gd_p\boldsymbol{u}):\gd(\boldsymbol{u}^*-\boldsymbol{u})\\
\le\int_0^t\left(\boldsymbol{f}^*-\bff\right)\cdot\left(\boldsymbol{u}^*-\boldsymbol{u}\right)+\frac12\int_\Omega|
\boldsymbol{h}^*-\boldsymbol{h}|^2.
\end{multline}

In the case $p\ge 2$, since
$$\int_0^t\int_\Omega
(\gd_p\boldsymbol{u}^*-\gd_p\boldsymbol{u}):\gd(\boldsymbol{u}^*-\boldsymbol{u})\ge C_p\int_0^t\int_\Omega
|\gd(\boldsymbol{u}^*-\boldsymbol{u})|^p,$$  the conclusion
follows easily by using H\"{o}lder and Poincaré inequalities.

In the case $1<p<2$, from (\ref{mutumuiep}) we find
$$\int_\Omega|\boldsymbol{u}^*(t)-\boldsymbol{u}(t)|^2\le \eps^*+\int_0^t\int_\Omega|\boldsymbol{u}^*-\boldsymbol{u}|^2,$$
which, by Gronwall inequality yields, first
$$\sup_{0<t<T}\int_\Omega|\boldsymbol{u}^*(t)-\boldsymbol{u}(t)|^2\le e^T\eps^*$$
and, afterwards
\begin{equation}
\int_0^T\int_\Omega\left(\gd_p\boldsymbol{u}^*-\gd_p\boldsymbol{u}\right):\gd(\boldsymbol{u}^*-\boldsymbol{u})\le\frac12(1+Te^T)\eps^*.
\end{equation}
Next we consider the following reverse H\"{o}lder inequality: given
$0<r<1$ and $r'=\frac{r}{r-1}$, if $F\in L^r(\Omega)$, $FG\in
L^1(\Omega)$ and $\displaystyle{\int_\Omega|G(x)|^{r'}dx<\infty}$ in
$\Omega_T$, one has
\begin{equation*}
\left(\int_\Omega|F(x)|^rdx\right)^{\frac1r}\le
\left(\int_\Omega|F(x)G(x)|dx\right)\left(\int_\Omega|G(x)|^{r'}dx\right)^{-\frac1{r'}}.
\end{equation*}
Letting $r=\frac{p}2$ and, for $i=1,\ldots,N$,
$F=|\gd(u^*_i-u_i)|^2$ we get
\begin{equation*}
\begin{array}{l}
\displaystyle{\int_{\hat{\Omega}_t^i}
(\gd_p u^*_i-\gd_p u_i)\cdot\gd(u^*_i-u_i)}\ge\displaystyle{\int_{\hat{\Omega}_t^i}\frac{|\gd(u^*_i-u_i)|^2}
{(|\gd u^*_i|+|\gd u_i|)^{2-p}}}\vspace{3mm}\\
\hspace{2cm}
\displaystyle{\ge\left(\int_{\hat{\Omega}_t^i}|\gd(u^*_i-u_i)|^p\right)^{\frac2p}
\left(\int_{\hat{\Omega}_t^i}\big(|\gd u^*_i|+|\gd
u_i|\big)^p\right)^{\frac{p-2}p}},
\end{array}
\end{equation*}
where $\hat{\Omega}_t^i=\{(x,t)\in\Omega_t:|\gd u^*_i|+|\gd
u_i|>0\}$. Thus, if we denote
$$\alpha_p\ge \left(\int_{\Omega_T}\big(|\gd u^*_i|+|\gd
u_i|\big)^p\right)^{\frac{2-p}p},$$ by (\ref{mutumuiep}), we conclude (\ref{p<2}) from
\begin{equation*}
\sum_{i=1}^N\Big(\int_{\Omega_T}|\nabla(u^*_i-u_i)|^p\Big)^{\frac2p}\\
\le\alpha_p\sum_{i=1}^N\int_{\Omega_T}(\gd_p u^*_i-\gd_p u_i)\cdot\gd(u^*_i-u_i)\le \alpha_p\,c\,\eps^*.
\end{equation*}
\end{proof}

\vspace{3mm}

\subsection{H\"{o}lder continuity and further regularity of the solution}

The regularity of the variational solutions of the evolution
$N$-membranes problem does not, in general, yield their boundedness
for $1<p\leq d$; but, by Sobolev imbedding, the solutions are
bounded for $p>d$ and even H\"{o}lder continuous in the space
variables for each $t \in (0,T)$.

However, estimates (\ref{215}) and (\ref{delta}) imply that, in
fact, $P\boldsymbol{u}$ has the same regularity in $\Omega_T$ as the
data $\bff$. Then, if $\bff\in\boldsymbol{L}^\infty(\Omega_T)$,
local and global H\"{o}lder estimates for the evolution $p$-Laplace
equation may be directly applied to bounded solutions of the
$N$-membranes problem (see \cite{livro do DiBenedetto}, \cite{lib}
or \cite{duv}). In order to illustrate these results, we assume in
addition that $\boldsymbol{h} \in \boldsymbol{L}^\infty (\Omega)$,
which also implies that $\boldsymbol{u} \in \boldsymbol{L}^\infty
(\Omega_T)$,
 and consequently that $\boldsymbol{u}$ and $\nabla
\boldsymbol{u}$ are locally H\"{o}lder continuous. Referring to
\cite{chen-dib} and \cite{lib} for the boundary and initial
regularity in the space of H\"{o}lder continuous functions
$C^\alpha$, $0<\alpha <1$, with the standard parabolic norms, we may
state the following result:

\begin{thm}
Suppose $\bff\in \boldsymbol{L}^\infty (\Omega)$ and the initial
data $\boldsymbol{h} \in \boldsymbol{C}^\alpha
(\overline{\Omega})\,\cap\,\K$, $0< \alpha <1$. Then the solution
$\boldsymbol{u} \in \boldsymbol{C}^{\alpha^\prime}
(\overline{\Omega}_T)$, $0<\alpha^\prime \leq \alpha <1$, and
$\nabla \boldsymbol{u} \in \boldsymbol{C}_{\mathrm{loc}}^{\beta}
(\overline{\Omega}_T)$, for some $0<\beta<1$. If, in addition,
$\partial \Omega \in \boldsymbol{C}^{1,\beta}$ and $\nabla
\boldsymbol{h} \in \boldsymbol{C}^\beta (\overline{\Omega})$,
$0<\beta<1$, then also $\nabla \boldsymbol{u} \in
\boldsymbol{C}^{\beta^\prime} (\overline{\Omega}_T)$, for some
$0<\beta^\prime \leq \beta<1$.
\end{thm}

In case of a linear operator ($p=2$) we can apply directly
Solonnikov's parabolic estimates (see \cite{lsu}, Thm. 9.1 of page
341).

\begin{thm}
Let $p=2$. Then, for any $\bff\in\boldsymbol{L}^q(\Omega_T)$, $q\ge 2$, the solution $\bu$ to
{\em (\ref{def3})-(\ref{ineq_var})}  satisfies
 $\boldsymbol{u} \in
\boldsymbol{W}_{q, \mathrm{loc}}^{2,1} (\Omega_T)$, which implies,
by Sobolev imbeddings, that $\boldsymbol{u}$ and $\nabla
\boldsymbol{u}$ are locally H\"{o}lder continuous, respectively for
$q>\frac{d+2}d$ and $q>d+2$. If, in addition, $\bs h\in \K\cap\bs
W^{2-\frac2q,q}(\Omega)$, those results can be extended up to the
boundary $\partial \Omega \in C^2$ and up to $t=0$, i.e.,
$\boldsymbol{u} \in \boldsymbol{W}_{q}^{2,1} (\Omega_T)$ and $\bu$,
$\gd\bu$ are H\"{o}lder continuous on $\overline{\Omega}_T$.
\end{thm}

\vspace{5mm}

\section{The $N$-system and its stability}

The $N$-membranes problem can, a posteriori, be regarded as a lower
obstacle problem for $u_1$, a double obstacle problem for $u_j$,
$j=2\le j\le N-1$, and an upper obstacle problem for $u_N$. This
fact has interesting consequences and, similarly to the theory of
the obstacle problem that we recall briefly for completeness, allows
us to characterize the $N$-membranes problem as a nonlinear
parabolic system with known discontinuous nonlinearities on the
right hand side as in (\ref{react_diff}).

\vspace{3mm}

\subsection{Dual estimates for obstacle type problems}

We consider the scalar two-\mbox{-obstacles} problem for the nonlinear
operator $P$ defined in (\ref{operator}), with compatible
Cauchy-Dirichlet data on $\partial_p\Omega_T$. Let
\begin{equation}
\label{cvarphi}
\varphi\in L^{p'}(\Omega_T),\qquad\eta\in W^{1,p}_0(\Omega))\cap L^2(\Omega),
\end{equation}
\begin{equation}
\label{cpsi}
\psi_1,\,\psi_2\in L^p(0,T;W^{1,p}(\Omega)),\quad\psi_1\ge\psi_2\ \mbox{ in }\Omega_T,\quad
\psi_1\ge 0\ge\psi_2\mbox{ on }\Sigma_T,
\end{equation}
and, for $j=1,2$,
\begin{equation}
\label{cpsij}
\partial_t\psi_j\in L^{p'}(0,T;W^{-1,p'}(\Omega)),\quad P\psi_j\in L^{p'}(\Omega_T),\quad \psi_1(0)
\ge \eta\ge \psi_2(0)\mbox{ on }\Omega_0.
\end{equation}

Using the Lipschitz continuous function $\theta_\eps$ defined in subsection
\ref{22} for each $\eps>0$, we may easily show that the problem
\begin{alignat}{1}
\label{pweps}
&Pw^\eps+\zeta_2\theta_\eps(w^\eps-\psi_2)-\zeta_1\theta_\eps(\psi_1-w^\eps)=\varphi\quad\mbox{ in }\Omega_T,\\
\label{eta}
&w^\eps=0\quad\mbox{ on }\Sigma_T\qquad\mbox{ and }\qquad w^\eps=\eta\quad\mbox{ on }\Omega_0,
\end{alignat}
where $\zeta_1=(P\psi_1-\varphi)^-$ and
$\zeta_2=(P\psi_2-\varphi)^+$, has a unique solution $w^\eps\in
L^p(0,T;W^{1,p}_0(\Omega))\cap C([0,T];L^2(\Omega))$, with
$Pw^\eps\in L^{p'}(\Omega_T)$, uniformly in $\eps\le 1$. Similarly
to Proposition \ref{prop22}, it is easy to show that
$$\psi_2-\eps\le w^\eps\le \psi_1+\eps\qquad\mbox{ a.e. in }\Omega_T,$$
and, when $\eps\rightarrow 0$, as in Theorem \ref{teo24}, that
$w^\eps\rightarrow w$ strongly in $L^p(0,T;W^{1,p}_0(\Omega))$,
where $w$ is the unique solution of the double obstacle problem
\begin{equation}\label{k2ob}
w\in\K_{\psi_2}^{\psi_1}=\{v\in L^p(0,T; W^{1,p}_0(\Omega)):\psi_1\ge v\ge \psi_2\mbox{ in }\Omega_T\},
\end{equation}
\begin{equation}\label{iv2ob}
\int_{\Omega_T}(Pw-\varphi)(v-w)\ge 0,\qquad\forall\,v\in\K_{\psi_2}^{\psi_1},\qquad\mbox{ a.e. }t\in(0,T),
\end{equation}
such that $w(0)=\eta$ on $\Omega$. The solution $w$ satisfies also
$$w\in L^p(0,T;W^{1,p}_0(\Omega))\cap C([0,T];L^2(\Omega))\qquad\mbox{ and }\qquad Pw\in L^{p'}(\Omega_T)$$
and, arguing as in Proposition 4.1 of \cite{rs}, we can state the
following important property.
\begin{prop}
The solution $w$ to {\em (\ref{k2ob})-(\ref{iv2ob})}, under
assumptions {\em (\ref{cvarphi})-(\ref{cpsij})}, satisfies the
parabolic nonlinear equation
\begin{equation}
Pw=\varphi+(P\psi_2-\varphi)^+
\chi_{\{w=\psi_2\}}-(P\psi_1-\varphi)^-
\chi_{\{w=\psi_1\}}\qquad\mbox{ a.e. in }\Omega_T.
\end{equation}

In addition, we have the Lewy-Stampacchia inequalities
\begin{equation}\label{2obs}
\varphi-(P\psi_1-\varphi)^-=\varphi\wedge P\psi_1\le Pw\le \varphi\vee P\psi_2=
\varphi+(P\psi_2-\varphi)^+\qquad\mbox{ a.e. in }\Omega_T
\end{equation}
and the a.e. in $\Omega_T$ necessary conditions for contact with the obstacles
\begin{equation}
\{w=\psi_1\}\subset\{P\psi_1\le\varphi\}\qquad\mbox{ and }\qquad\{w=\psi_2\}\subset\{P\psi_2\ge\varphi\}
\end{equation}
being the inclusions valid up to subsets of $\Omega_T$ with zero measure.
\end{prop}

\begin{rem} We note that for the case of only one-obstacle, we have similar properties. In fact,
if we formally take $\psi_1\equiv+\infty$, we have a lower obstacle problem
\begin{equation}\label{lobs}
w\ge\psi_2\qquad\mbox{ and }\qquad\varphi\le Pw\le\varphi\vee P\psi_2\qquad\mbox{ a.e. in }\qquad\Omega_T,
\end{equation}
and, with $\psi_2\equiv-\infty$, an upper obstacle problem
\begin{equation}\label{uobs}
w\le\psi_1\qquad\mbox{ and }\qquad\varphi\wedge P\psi_1\le Pw\le\varphi\qquad\mbox{ a.e. in }\qquad\Omega_T.
\end{equation}
\end{rem}

Analogously, the semilinear equation holds in each case with the
corresponding characteristic function, respectively.

\begin{rem}
As observed in {\em \cite{rs}}, we have
\begin{equation}
Pw=\varphi\qquad\mbox{ a.e. in }\{\psi_2<w<\psi_1\}
\end{equation}
and due to the fact that both $Pw$ and $P\psi_j$ are integrable, we have
\begin{equation}
Pw=P\psi_j\qquad \mbox{ a.e. in }\{w=\psi_j\}\qquad\mbox{ for }j=1,2.
\end{equation}
\end{rem}

Using the regularity of Theorem \ref{teo24}, we easily see that each
component $u_i$ of the $N$-membranes problem solves an obstacle type
problem (\ref{k2ob})-(\ref{iv2ob}) with $\varphi=f_i$,
$\psi_1=u_{i-1}$ and $\psi_2=u_{i+1}$ (with the conventions
$u_0\equiv+\infty$ and $u_{N+1}\equiv-\infty$ corresponding to the
one-obstacle problems). Hence, we have from (\ref{lobs}),
(\ref{2obs}) and (\ref{uobs}), respectively, a.e. in $\Omega_T$,
\begin{eqnarray*}
\begin{array}{cllll}
 f_1&\le&P u_1&\le& f_1\vee Pu_2\qquad\qquad\\
\vdots& &&&\vdots\\
 f_i\wedge Pu_{i-1}&\le&Pu_i&\le&f_i\vee Pu_{i+1}\qquad(i=2,\ldots,N-1)\\
\vdots& &&&\vdots\\
   f_N\wedge Pu_{N-1}&\le&Pu_{N}&\le& f_N\qquad\mbox{ a.e. in }\Omega_T.\end{array}
\end{eqnarray*}

By simple iteration, we have shown the following Lewy-Stampacchia
type inequalities, that extend Theorem 3.5 of \cite{ars1} to the
evolution $N$-membranes problem

\begin{thm} \label{lsineq} The solution $\bu$ of {\em (\ref{def3})-(\ref{ineq_var})} satisfies
$$\bigwedge_{j=1}^i f_j\le Pu_i\le\bigvee_{j=i}^N f_j\qquad\mbox{ a.e. in }\Omega_T,\qquad i=1,\ldots,N.$$
\end{thm}

\vspace{3mm}

\subsection{The nonlinear $N$-system}

As a consequence of the equivalence of the $N$-\mbox{-membranes} inequality
with two one--obstacle problems and $N-2$  two--obstacles problems,
we may prove the equivalence of this inequality with a $N$-system of
equations, strongly coupled by the $\frac{N(N-1)}2$ coincidence sets
$I_{j,k}$ defined in (\ref{coinc_sets}). Indeed, we can argue as in
section 4 of \cite{ars1}, and since we know that $Pu_i\in
L^{p'}(\Omega_T)$, for all $i=1,\ldots,N$, we have on each
coincidence set
$$Pu_j=\cdots=Pu_k\qquad\mbox{ a.e. in }\qquad I_{j,k}=\{u_j=\cdots=u_k\}$$
and we conclude, for each $j\le i\le k$,
$$Pu_i=\langle\bff\rangle_{j,k}\qquad\mbox{ a.e. in }\qquad I_{j,k},$$
where we introduce the averages of $\bff$ by
$$\langle\bff\rangle_{j,k}=\frac{f_j+\cdots+f_k}{k-j+1},\qquad 1\le j\le k\le N.$$

On the other hand, in the complementary sets $\Omega_T\setminus I_{j,k}$, for each $i>k>j$ or
$i<j<k$, we have
$$Pu_i=f_i\qquad\mbox{ a.e. in }\qquad\Omega_T\setminus I_{j,k},$$
and we conclude, as in \cite{ars1}, the following explicit form for
(\ref{react_diff}).

\begin{thm} The variational solution of the $N$-membranes problem {\em (\ref{def3})-(\ref{ineq_var})} satisfies the system
{\em (}$i=1,\ldots,N${\em )}
\begin{equation}\label{systemf}
Pu_i=f_i+\sum_{\mbox{\scriptsize{$1\le j<k\le N, j\le i\le
k$}}}b^{j,k}_i[\bff]\,\cchi_{j,k}\qquad\mbox{ a.e. in }\Omega_T,
\end{equation}
where $\cchi_{j,k}$ denotes the characteristic function of each $I_{j,k}$ and
\begin{eqnarray*}
 b_i^{j,k}[\bff] & = & \left\{
\begin{array}{ll}
\langle \bff\rangle_{j,k}-\langle \bff\rangle_{j,k-1} & \mbox{ if
 }\ i=j\\
 & \\
\mbox{} \langle \bff\rangle_{j,k}-\langle \bff\rangle_{j+1,k} & \mbox{ if }\ i=k\\
  & \\
\frac{2}{(k-j)(k-j+1)}\left(\langle \bff\rangle_{j+1,k-1}
-\frac12(\bff_j+\bff_k)\right) & \mbox{ if }\ j<i<k.
\end{array}\right.
\end{eqnarray*}
\end{thm}

For the particular case $N=3$ (and $N=2$), we can easily deduce
(\ref{sistemachi}) from (\ref{systemf}).

\vspace{3mm}

\subsection{Convergence of coincidence sets}

From Theorem \ref{tap}, we know that if for sequences
\begin{equation}\label{convf}
\tende{\bff^\nu}{\nu}{\bff}\qquad\mbox{ in }\bs L^q(\Omega_T),\qquad q=p'\wedge 2,
\end{equation}
\begin{equation}\label{convh}
\tende{\bs h^\nu}{\nu}{\bs h}\qquad\mbox{ in }\bs L^2(\Omega_T),
\end{equation}
then, the corresponding solutions of  (\ref{def3})-(\ref{ineq_var}) also converge
\begin{equation}\label{convu}
\tende{\bu^\nu}{\nu}{\bu}\qquad\mbox{ in }C^0([0,T];\bs
L^2(\Omega))\cap L^p(0,T;\bs W^{1,p}_0(\Omega)).
\end{equation}

Consequently, we have
$$\Delta_p\bu^\nu\ _{ \stackrel{\lraup}{\nu}}\Delta_p\bu\qquad\mbox{ in }L^{p'}(0,T;\bs W^{-1,p'}(\Omega))-\mbox{weak}$$
and, by Theorem \ref{lsineq}, also
$$P\bu^\nu\ _{ \stackrel{\lraup}{\nu}}\ P\bu\qquad\mbox{ in }\bs L^{q}(\Omega_T)-\mbox{weak}.$$

Since the characteristic functions
$\cchi_{j,k}^\nu=\cchi_{\{u_j^\nu=\cdots=u_k^\nu\}}$ satisfy $0\le
\cchi_{j,k}^\nu\le 1$ a.e. in $\Omega_T$, there are
$\cchi_{j,k}^*\in L^\infty(\Omega_T)$ such that
$$\cchi_{j,k}^\nu\ _{ \stackrel{\lraup}{\nu}}\ \cchi_{j,k}^*\qquad\mbox{ in }\bs L^{\infty}(\Omega_T)-\mbox{weak}*.$$

Passing to the limit in
$$Pu_i^\nu=f_i^\nu+\sum_{\mbox{\scriptsize{$1\le j<k\le N,
j\le i\le k$}}}b^{j,k}_i[\bff^\nu]\,\cchi_{j,k}^\nu,$$
we obtain, for each $i=1,\ldots,N,$
$$Pu_i=f_i+\sum_{\mbox{\scriptsize{$1\le j<k\le N,
j\le i\le k$}}}b^{j,k}_i[\bff]\,\cchi_{j,k}^*,$$
which compared with (\ref{systemf}) yields
$$\sum_{\mbox{\scriptsize{$1\le j<k\le N,
j\le i\le
k$}}}b^{j,k}_i[\bff](\cchi_{j,k}-\cchi_{j,k}^*)=0\qquad\mbox{ a.e.
in }\Omega_T.$$

Arguing exactly as in the proof of Theorem 4.6 of \cite{ars1}, we
conclude, under the same nondegeneracy assumption for the limit
data, namely
\begin{equation}\label{nondeg}
\langle\bff\rangle_{i,j}\neq\langle\bff\rangle_{j+1,k}, \ \mbox{a.e.
in }\Omega_T, \ \mbox{ for all }i,j,k\in\{1,\ldots,N\},\mbox{ with
}i\le j\le k,
\end{equation}
that $\cchi_{j,k}=\cchi_{j,k}^*$ and prove the following stability
property for the respective coincidence sets
$I_{j,k}^\nu=\{u_j^\nu=\cdots=u_k^\nu\}$.

\begin{thm}\label{convcc}
Under the convergence assumptions {\em (\ref{convf})} and {\em
(\ref{convh})}, the characteristic functions associated with the
convergent variational solutions {\em (\ref{convu})} also converge
$$\tende{\cchi_{\{u_j^\nu=\cdots=u_k^\nu\}}}{\nu}{\cchi_{\{u_j=\cdots=u_k\}}}\qquad\mbox{ in }L^s(\Omega_T),$$
for any $1\le s<\infty$, all $1\le j<k\le N$, provided the
nondegeneracy condition {\em (\ref{nondeg})} holds.
\end{thm}

\vspace{3mm}

\subsection{Asymptotic stabilization as $t\rightarrow\infty$}

In this section we assume $p\ge 2$ and we consider the unique
solution $\bu^\infty$ to the stationary $N$-membranes problem for a
given $\bff^\infty\in\bs L^{p'}(\Omega)$:
\begin{equation}
\bu^\infty\in\K:\qquad\int_\Omega\gd_p\bu^\infty:\gd(\bs v-\bu^\infty)\ge\int_\Omega
\bff^\infty\cdot(\bs v-\bu^\infty),\qquad\forall\bs v\in\K.
\end{equation}

Supposing that the problem  (\ref{def3})-(\ref{ineq_var}) is
solvable for all $T<\infty$ and that
$\bff(t)\longrightarrow\bff^\infty$ in $\bs L^{p'}(\Omega)$ as
$t\rightarrow\infty$ in the sense
\begin{equation}\label{tt+1f}
\int_t^{t+1}\int_\Omega|\bff(t)-\bff^\infty|^{p'}\longrightarrow 0\qquad\mbox{ as }t\rightarrow\infty,
\end{equation}
by the results of \cite{simon}, the evolutive solution $\bu(t)$ is such that
\begin{equation}\label{tt+1u}
\tende{\bu(t)}{t\rightarrow\infty}{\bu^\infty}\qquad \mbox{ in }\bs L^2(\Omega),
\end{equation}
\begin{equation}\label{tt+1gdu}
\int_t^{t+1}\int_\Omega|\gd\bu(t)-\gd\bu^\infty|^{p}\longrightarrow 0\qquad\mbox{ as }t\rightarrow\infty.
\end{equation}

By the results of \cite{ars1}, the stationary solution also solves
the nonlinear $N$-system
\begin{equation}
-\Delta_p u^\infty_i=f^\infty_i+\sum_{\mbox{\footnotesize{$1\le j<k\le N,\
j\le i\le k$}}}b^{j,k}_i[\bff^\infty]\,\cchi_{j,k}^\infty\qquad\mbox{ a.e. in }\Omega,
\end{equation}
where $\cchi_{j,k}^\infty=\cchi_{\{u_j^\infty=\cdots=u_k^\infty\}}$
denotes the characteristic function of the limit coincidence set
$I_{j,k}^\infty=\{x\in\Omega:u_j^\infty(x)=\cdots=u_k^\infty(x)\}.$

Denoting by $\cchi_{j,k}(t)$ the characteristic functions of
$I_{j,k}(t)=\{u_j(t)=\cdots=u_k(t)\}$ at time $t$, we have the
following asymptotic convergence result as $t\rightarrow\infty$.

\begin{thm}
Under assumption {\em (\ref{tt+1f})}, the variational solution of
the evolution $N$-membranes problem converges to the corresponding
stationary solution in the sense {\em (\ref{tt+1u})} and {\em
(\ref{tt+1gdu})}. In addition, the characteristic functions satisfy
\begin{equation}\label{convforte}
\cchi_{j,k}(t)\longrightarrow\cchi_{j,k}^\infty\qquad\mbox{ as }t
\rightarrow\infty\quad\mbox{ in }L^s(\Omega),
\end{equation}
for any $1\le s<\infty$, for all $1\le j<k\le N$, provided we assume
\begin{equation}\label{nondeginf}
\langle\bff^\infty\rangle_{i,j}\neq\langle\bff^\infty\rangle_{j+1,k}\qquad\mbox{a.e.
in }\Omega,\quad \mbox{ for all }1\le i\le j<k\le N.
\end{equation}
\end{thm}
\begin{proof}
We rewrite (\ref{tt+1gdu}) for $\bs w(t)=\bu(t)-\bu^\infty$ as
$$\int_t^{t+1}\int_\Omega|\gd\bs w(\tau)|^pd\tau=\int_0^1\int_\Omega|\gd\bs w(t+\sigma)|^pd\sigma
\longrightarrow 0\qquad\mbox{ as }t\rightarrow\infty,$$ and this
convergence can be interpreted as
\begin{equation}\label{sharp}
\bs w_\sharp(t)\longrightarrow 0\qquad\mbox{ as }t\rightarrow\infty\mbox{ in }L^p(0,1;\bs W^{1,p}_0(\Omega)),
\end{equation}
where we define $\bs w_\sharp\in L^\infty(0,\infty;L^p(0,1;\bs W^{1,p}_0(\Omega)))$ as
$$\bs w_\sharp(t) (\sigma,\cdot)= \bs w(t+\sigma,\cdot) \in \bs W^{1,p}_0(\Omega), \quad \sigma \in (0,1).$$

Consequently, from (\ref{sharp}) we have
$$\Delta_p\bs u_\sharp(t)\rightharpoonup \Delta_p\bs u_\sharp^\infty\qquad\mbox{ as }t\rightarrow
\infty \mbox{ in }L^{p'}(0,1;\bs W^{-1,p'}(\Omega))-\mbox{weak}$$
and, recalling the estimates of Theorem \ref{lsineq} and the
assumption (\ref{tt+1f}), we may conclude
$$(\partial_t\bu_\sharp-\Delta_p\bs u_\sharp)(t)\rightharpoonup -\Delta_p\bs
u_\sharp^\infty\qquad\mbox{ as }t\rightarrow\infty \mbox{ in }L^{p'}(0,1;\bs L^{p'}(\Omega))-\mbox{weak}.$$

Since $\bu_\sharp(t)$ solves (\ref{systemf}), a.e. in $\Omega$ and
for a.e. $t>0$, we can pass to the limit, as $t\rightarrow\infty$,
in $L^{p'}(0,1; \bs L^{p'}(\Omega))$. As in the proof of Theorem
\ref{convcc} (and Theorem 4.6 of \cite{ars1}), we conclude that
assumption (\ref{nondeginf}) implies the convergence
$\cchi_{j,k}(t)\longrightarrow\cchi_{j,k}^\infty$ as $t\rightarrow
\infty$, first as functions of $L^\infty(0,1;L^\infty(\Omega))$ with
the weak-$*$ topology and, afterwards, also in the sense of
(\ref{convforte}). Indeed, since they are characteristic functions
and any subsequence of $\cchi_{j,k}(t)$ has the same limit
$\cchi_{j,k}^\infty$, their weak convergence implies the strong
convergence in $L^s(\Omega)$ for all $1\le s<\infty$.
\end{proof}

\end{document}